\documentclass{amsart}

\usepackage{amsmath,latexsym}
\usepackage{amsfonts}
\usepackage{amssymb,enumerate}
\usepackage{amsthm}
\usepackage[all]{xy}
\usepackage{hyperref}

\theoremstyle{plain}

\newtheorem{syzthm}{Theorem}
\newtheorem{syzprop}[syzthm]{Proposition}
\newtheorem{syzcor}[syzthm]{Corollary}
\newtheorem{syzlem}[syzthm]{Lemma}

\theoremstyle{definition}
\newtheorem{syzQ}[syzthm]{Question}
\newtheorem{syzrmk}[syzthm]{Remark}
\newtheorem{syzfact}[syzthm]{Fact}
\newtheorem{syzex}[syzthm]{Example}
\newtheorem{syzconj}[syzthm]{Conjecture}

\theoremstyle{definition}

\newcommand{\LC}{\operatorname{H}}

\newcommand{\pd}{\operatorname{pd}}

\newcommand{\depth}{\operatorname{depth}}

\newcommand{\len}{\lambda}

\newcommand{\coker}{\operatorname{Coker}}
\newcommand{\spec}{\operatorname{Spec}}

\newcommand{\im}{\operatorname{Im}}

\newcommand{\ideal}[1]{\mathfrak{#1}}
\newcommand{\m}{\ideal{m}}
\newcommand{\n}{\ideal{n}}
\newcommand{\p}{\ideal{p}}
\newcommand{\q}{\ideal{q}}

\newcommand{\supp}{\operatorname{Supp}}

\newcommand{\minn}{\operatorname{Min}}
\newcommand{\height}{\operatorname{height}}

\renewcommand{\geq}{\geqslant}
\renewcommand{\leq}{\leqslant}

\newcommand{\Tor}[4][R]{\operatorname{Tor}^{#1}_{#2}(#3,#4)}

\author{Kristen A. Beck}

\address{Kristen A. Beck, Department of Mathematics,
University of Texas at Arlington
P.O. Box 19408,
Arlington, TX 76019-0408, USA}

\email{kbeck@uta.edu}
 
\author{Micah J. Leamer}

\address{Micah J. Leamer, Department of Mathematics,
University of Nebraska-Lincoln,
PO Box 880130,
Lincoln, NE 68588-0130,
USA}

\email{s-mleamer1@math.unl.edu}

\thanks{This material is based on work that began at the 2011 Mathematical Research Community in Commutative Algebra, located in Snowbird, UT.  The MRC was funded by the American Mathematical Society and the National Science Foundation. Kristen Beck was partially supported by NSA Grant H98230-07-1-0197.
Micah Leamer was funded in part by a GAANN grant from the Department of Education. Part of this work also appears in Micah Leamer's Ph.D. thesis.}

\title{Asymptotic behavior of dimensions of syzygies}

\begin{document}

\begin{abstract}
Let $R$ be a commutative noetherian local ring, and $M$ a finitely generated $R$-module of infinite projective dimension.  It is well-known that the depths of the syzygy modules of $M$ eventually stabilize to the depth of $R$.  In this paper, we investigate the conditions under which a similar statement can be made regarding dimension.  In particular, we show that if $R$ is equidimensional and the Betti numbers of $M$ are eventually non-decreasing, then the dimension of any sufficiently high syzygy module of $M$ coincides with the dimension of $R$.
\end{abstract}

\maketitle

\section*{Introduction}

Throughout this paper $R$ will denote a commutative noetherian local ring with identity element, unique maximal ideal $\m$ and residue field $k:=R/\m$. Let $M$ be a finitely generated $R$-module. The $i$th Betti number of $M$ is given by $\beta_i(M):=\dim_k(\Tor{i}{k}{M})$. A minimal free resolution of $M$ then has the form
\[
\quad\xymatrix{ \cdots \ar[r]^{\delta_3}& R^{\beta_2(M)}\ar[r]^{\delta_2} & R^{\beta_1(M)} \ar[r]^{\delta_1} & R^{\beta_0(M)}\ar[r] & 0}.
\]
The $i$th syzygy module of $M$ is $\Omega_i(M):=\coker(\delta_{i+1})$. We let $\minn(M)$ denote the set of minimal elements under inclusion of $\supp(M):=\{\p\in\spec(R)|\ M_\p\neq 0\}$.  

Our main result is the following, which is part of Theorem \ref{thm:main}.
\begin{syzthm}
Let $R$ be a noetherian local ring and $M$ a finitely generated $R$-module with eventually non-decreasing Betti numbers. Then for all $i\gg 0$ we have $\minn(\Omega_i(M))\subseteq\minn(R)$ and $\supp(\Omega_{i}(M))=\supp(\Omega_{i+2}(M))$.
\end{syzthm}
An important consequence of Theorem \ref{thm:main} is the following corollary, which follows  immediately from Corollary 
\ref{cor:dim}.
\begin{syzcor}
\label{cor:equidim}
Let $R$ be an equidimensional noetherian local ring and $M$ a finitely generated $R$-module with eventually non-decreasing Betti numbers. Then the sequence $(\dim(\Omega_i(M)))_{i=0}^{\infty}$ is constant for $i\gg 0$.
\end{syzcor}
This raises the following open question.
\begin{syzQ}\label{Q1}
Let $R$ be a noetherian local ring and $M$ a finitely generated $R$-module. Is $(\dim(\Omega_i(M)))_{i=0}^{\infty}$ constant for all $i\gg 0$?
\end{syzQ}
This question was also explored in the last section of \cite{Crabbe}. In \cite[Remark 5.2 (i)]{Crabbe} it is noted that if $R$ is unmixed and equidimensional,  then $(\dim(\Omega_i(M)))_{i=0}^{\infty}$ is constant for $i\gg 0$. This is clear since the associated primes of any submodule of $R^{\beta_i(M)}$ are also associated primes of $R$, and are therefore primes of maximal dimension by assumption.

It is worth noting that the asymptotic behavior of the depths of syzygy modules is known.  Given $M$ the sequence $(\depth(\Omega_i(M)))_{i=0}^{\infty}$ is constant for all $i\gg 0$.
Let $\pd(M)$ denote the  projective dimension of $M$. In particular if $\pd(M)=\infty$, then
$\depth(\Omega_n(M))\geq\depth(R)$ for $n\geq\max\{0,\depth(R)-\depth(M)\}$, with at most one strict inequality at either $n=0$ or $n=\depth(R)-\depth(M)+1$; see 
\cite[Proposition 10]{Okiyama} or \cite[Proposition 1.2.8]{Avramov2}.  It follows therefore that if $\pd(M)=\infty$ and $R$ is Cohen-Macaulay, then $\dim(\Omega_n(M))=\dim(R)$ for $n\gg 0$.

All of our results are for modules whose Betti numbers are eventually non-decreasing. Therefore finding a proof for the following conjecture of L. Avramov would improve our results.
\begin{syzconj}\label{Q2}\cite{Avramov1}
The Betti numbers of any finitely generated module over an arbitrary noetherian  local ring are eventually non-decreasing.
\end{syzconj}
There are a plethora of cases for which this conjecture is known to be true. 
J. Lescot \cite[Corollaire 6.5]{Lescot} showed that over a Golod ring, which is not a hypersurface any finitely generated module of infinite projective dimension will have eventually increasing Betti numbers. Also L.-C. Sun  \cite[Corollary]{Sun} showed that over rings of codepth less than or equal to three and Gorenstein rings of codepth four all finitely generated modules have eventually non-decreasing Betti numbers. Several other interesting  cases are also proven in \cite{Avramov3}, \cite{Choi} and \cite{Sun2}. 
 
Whenever the Betti numbers of a module are eventually strictly increasing it is known that the the dimension of a sufficiently high syzygy will have the dimension of the ring. This is clear from the next lemma, which is mentioned without proof in 
\cite[Remark 5.2 (iii)]{Crabbe}.

\section*{Results}\label{sec:results}

 We denote the length of an $R$-module $M$ by $\lambda_R(M)$ or simply $\lambda(M)$ when the ring is unambiguous.

\begin{syzlem}
Let $R$ be a noetherian local ring and $M$ a finitely generated $R$-module.
If $\beta_i(M)>\beta_{i-1}(M)$ for some $i>0$, 
then $\supp(\Omega_{i+1}(M))=\spec(R)$; hence $\dim(\Omega_{i+1}(M))=\dim(R)$.
\end{syzlem}

\begin{proof}
Given $\q\in\spec(R)$ there exists $\mathfrak{p}\in\minn(R)$ such that $\p\subseteq \q$. Localizing the exact sequence
\[
0\to \Omega_{i+1}^R(M)\to R^{\beta_i(M)}\to R^{\beta_{i-1}(M)}
\]
at $\mathfrak{p}$  we obtain the following inequalities. 
\[
\lambda_{R_\p}(\Omega_{i+1}(M)_\p)\geq \lambda_{R_\p}(R_{\mathfrak{p}}^{\beta_{i}(M)})- \lambda_{R_\p}(R_{\mathfrak{p}}^{\beta_{i-1}(M)})=\lambda_{R_\p}(R_\p)(\beta_{i}(M)-\beta_{i-1}(M))>0
\]
Thus $\p\in\supp(\Omega_{i+1}(M))$; hence $\q\in\supp(\Omega_{i+1}(M))$ and the result follows.
\end{proof}

The following lemma is used in the proof of our main result, Theorem \ref{thm:main}.

\begin{syzlem} \label{lem:supp}
\index{Betti number!for local rings, $\beta^R_i(-)$}
Let $R$ be a noetherian local ring and $M$ a finitely generated $R$-module. For a given $n\in\mathbb{N}$ suppose that $\beta_0(M)\leq\beta_{1}(M)\leq \hdots\leq \beta_{2n-1}(M)$ and that $\supp(\Omega_{2n}(M))\neq\spec(R)$.  Then we have the following:
\begin{enumerate}[(a)]
\item  \label{lem:supp1} $\beta_{2i}(M)=\beta_{2i+1}(M)$  for $i=0,\hdots,n-1$;
\item  \label{lem:supp2} $\supp(\Omega_{2i+2}(M))\subseteq\supp(\Omega_{2i}(M))$ for $i=0,\hdots ,n-1$; and
\item  \label{lem:supp3} $\supp(\Omega_{2n}(M))\cap\minn(R)=\supp(M)\cap\minn(R)$.
\end{enumerate}
\end{syzlem}

\begin{proof}
 Choose $\p\in\minn(R)\diagdown\supp(\Omega_{2n}(M))$. Localizing part of a minimal free resolution of $M$ at $\p$, we get an exact sequence of finite-length $R_\p$-modules of the following form.
\[
0\to R_{\p}^{\beta_{2n-1}(M)}\xrightarrow{\varphi_{2n-1}} R_{\p}^{\beta_{2n-2}(M)}\to\cdots\to R_{\p}^{\beta_0(M)}\xrightarrow{\varphi_0} M_\p\to 0
\]
Since $\varphi_{2n-1}$ is an injection, $\len( R_{\p}^{\beta_{2n-2}(M)})\geq\len(R_{\p}^{\beta_{2n-1}(M)})$; hence $\beta_{2n-2}(M)\geq\beta_{2n-1}(M)$.  It follows that $\beta_{2n-2}(M)=\beta_{2n-1}(M)$. Since $R_\p$ has finite length it follows that $\varphi_{2n-1}$ is an isomorphism and $\varphi_{2n-2}$ is the zero map. By repeating this argument, one sees that $\varphi_{2i+1}$ is an isomorphism, $\varphi_{2i}$ is the zero map, and $\beta_{2i}(M)=\beta_{2i+1}(M)$ for each $i=0,1,\hdots, n-1$.  In particular, we have shown \eqref{lem:supp1}.

Since $\varphi_0$ is the zero map, $M_\p=0$; hence $\p\notin\supp(M)$.  It follows that 
\begin{equation}\label{eq:supp}
\supp(M)\cap\minn(R)\subseteq\supp(\Omega_{2n}(M))\cap\minn(R).
\end{equation}
 
Let $\q\in\spec(R)\diagdown\supp(\Omega_{2i}(M))$ for some $i$ with $0\leq i\leq n-1$. Localizing 
part of a minimal free resolution of $M$ at $\q$ we obtain an exact sequence of the following form.
\[
0\to \Omega_{2i+2}(M)_{\q}\to R_{\q}^{\beta_{2i+1}(M)}\to R_{\q}^{\beta_{2i}(M)}\to 0
\]
Since $\beta_{2i+1}(M)=\beta_{2i}(M)$ it follows that $\Omega_{2i+2}(M)_{\q}=0$ and $\q\notin\supp(\Omega_{2i+2}(M))$.  Thus $\supp(\Omega_{2i+2}(M))\subseteq \supp(\Omega_{2i}(M))$ for 
$ i=0,\hdots,n-1$, which proves \eqref{lem:supp2}. Consequently
$\supp(\Omega_{2n}(M))\cap\minn(R)\subseteq\supp(M)\cap\minn(R)$.
 Since \eqref{eq:supp} provides the reverse containment, \eqref{lem:supp3} is now immediate.
\end{proof}

We make the following fact explicit in order to clarify some of our argumentation.
\begin{syzfact} \label{fact:linalg}
Let $(R,\m)$ be a noetherian local ring. Let $B$ be an $n\times m$ matrix with entries in $R$ defining a map from $R^n$ to $R^m$. Applying invertible row and column operations to $B$ one can obtain a matrix $B'=I_h\oplus A$ where $I_h$ is the $h\times h$ identity matrix for some $h\geq 0$ and $A$ is an $(n-h)\times(m-h)$ matrix with entries in $\m$. 
\end{syzfact}

\begin{syzthm}\label{thm:main}
\index{Betti number!for local rings, $\beta^R_i(-)$}
Let $R$ be a noetherian local ring and $M$ a finitely generated $R$-module with eventually non-decreasing Betti numbers. Then for all $n\gg 0$ we have the following:
\begin{enumerate}[(a)]
\item \label{thm:main1} $\minn(\Omega_n(M))\subseteq\minn(R)$;
\item \label{thm:main2} $\supp(\Omega_{n}(M))=\supp(\Omega_{n+2i})$ for all $i\geq 0$; and
\item \label{thm:main3} if $\supp(\Omega_{n}(M))\neq\spec(R)$, then $\beta_{n+2i}(M)=\beta_{n+2i+1}(M)$ for all $i\geq 0$.
\end{enumerate}
\end{syzthm}

\begin{proof}
We may assume that $\pd(M)=\infty$. By replacing $M$ by a sufficiently high syzygy, one may assume that 
$\beta_{i+1}(M)\geq\beta_i(M)$ for all $i\geq 0$. Assuming $M$ was replaced by an even (odd) syzygy, if $\supp(\Omega_{2i}(M))=\spec(R)$ for $i\gg0$, then all of the statements hold for even (odd) syzygies. Therefore we may suppose that there exist infinitely many $i\in\mathbb{N}$ such that $\supp(\Omega_{2i}(M))\neq\spec(R)$. 

Since $\minn(R)$ is a finite set we may choose $\p\in\minn(R)$ so that there are infinitely many $i\in\mathbb{N}$ for which 
$\p\notin\supp(\Omega_{2i}(M))$. 
For each positive integer $c$ such that $\p\notin\supp(\Omega_{2c}(M))$, Lemma \ref{lem:supp} implies that we have
$\supp(\Omega_{2i+2}(M))\subseteq\supp(\Omega_{2i}(M))$ and $\beta_{2i}(M)=\beta_{2i+1}(M)$ for all $0\leq i<c$. Since, $c$ can be chosen to be arbitrarily large 
 we have
$\supp(\Omega_{2i+2}(M))\subseteq\supp(\Omega_{2i}(M))$ and $\beta_{2i}(M)=\beta_{2i+1}(M)$ for all $i\geq 0$.
Since closed sets in the Zariski topology satisfy the descending chain condition, it follows that we may choose $m\gg 0$ such that $(\supp(\Omega_{2m+2i}(M)))_{i=0}^{\infty}$ is constant, proving \eqref{thm:main2}. Therefore the assumption that there exist infinitely many $i\in\mathbb{N}$ such that $\supp(\Omega_{2i}(M))\neq\spec(R)$ is equivalent to assuming that $\supp(\Omega_{2m}(M))\neq\spec(R)$ and \eqref{thm:main3} follows.

Therefore it remains to show that $\minn(\Omega_{2i}(M))\subseteq\minn(R)$ for $i\gg 0$.
Choose $\q\in\minn(\Omega_{2m}(M))$. Let $S:=R_\q$, $M_i:=(\Omega_{2m+2i}(M))_{\q}$  for $i\geq 0$ and $\n:=\q R_\q$. Note that $\n$ is the maximal ideal for $S$.
For all $i\geq 0$ we obtain a commutative diagram of the form 
\begin{equation}
\begin{split}
\xymatrixrowsep{1pc}\xymatrixcolsep{1pc}\xymatrix
{
0\ar[r] & M_{i+1}\ar[r] & S^{b_i}\ar[rr]^{\alpha_i}\ar@{->>}[dr]_{\phi_i} && S^{b_i}\ar[r] & M_i\ar[r] & 0\\
&&& N_i\ar@{^{(}->}[ur]_{\psi_i}
}
\end{split}\label{commdiag}
\end{equation}
where the top row is exact and $N_i:=\im(\alpha_i)$.
If the matrix $A_i$ defining the map $\alpha_i:S^{b_i}\to S^{b_i}$ has some entries which are units, then by Fact \ref{fact:linalg} we can reduce this sequence by taking away free summands; hence we may assume that $A_i$ has all of its entries in $\n$. 

Let $\LC_\n^i(-)$ denote the $i$th local cohomology functor with respect to $\n$. 
 For background on local cohomology see \cite{Iyengar}. 
Since $M_{i+1}$ has finite length $\LC^0_{\n}(M_{i+1})\cong M_{i+1}$ and $\LC^j_{\n}(M_{i+1})=0$ for all $j>0$. From the long exact sequence of local cohomology modules associated to the short exact sequence
\[
\xymatrix{0\ar[r]& M_{i+1} \ar[r]& S^{b_i}\ar[r]^{\phi_i}& N_i\ar[r]& 0},
\]
 we get an exact sequence
\begin{equation}\label{longExact1}
\xymatrix{0\ar[r] & M_{i+1} \ar[r] & \LC^0_{\n}(S^{b_i})\ar[r] & \LC^0_{\n}(N_i)\ar[r] & 0}
\end{equation}
and isomorphisms $\LC^j_{\n}(\phi_i):\LC^j_{\n}(S^{b_i})\to \LC^j_{\n}(N_i)$ for all $j\geq 1$.
Also the exact sequence 
\[
\xymatrix{0\ar[r]& N_i\ar[r]^{\psi_i} & S^{b_i}\ar[r]& M_i\ar[r] &0}
\]
yields an exact sequence
\begin{equation}\label{longExact}
\xymatrixcolsep{1.5pc}\xymatrix{0\ar[r] & \LC^0_{\n}(N_{i}) \ar[r] & \LC^0_{\n}(S^{b_i})\ar[r] & M_i\ar[r]^{\gamma_i\ \ \  } & \LC^1_{\n}(N_{i}) \ar[r] & \LC^1_{\n}(S^{b_i})\ar[r] & 0} 
\end{equation}
and isomorphims $\LC^j_{\n}(\psi_i): \LC^j_{\n}(N_{i}) \to \LC^j_{\n}(S^{b_i})$ for all $j\geq 2$.
Here we are defining $\gamma_i:M_i\to \LC^1_{\n}(N_{i})$ to be the map found in exact sequence \eqref{longExact}.
By the additivity of length we get the first and third steps in the next display from sequences \eqref{longExact} and \eqref{longExact1} respectively.
\begin{align*}
\len(M_i)&=\len(\LC^0_{\n}(S^{b_i}))-\len(\LC^0_{\n}(N_i))+\lambda(\im(\gamma_i))\\
&\geq \len(\LC^0_{\n}(S^{b_i}))-\len(\LC^0_{\n}(N_i))\\
&=\len(M_{i+1})
\end{align*}

Since the sequence $(\len(M_i))_{i=0}^{\infty}$ is positive and non-increasing it is eventually constant.
Choose $\ell\in\mathbb{N}$ such that $\len(M_\ell)=\len(M_{\ell+1})$. Then 
$\lambda(\im(\gamma_{\ell})))=0$.  Therefore $\gamma_{\ell}$ is the zero map. 
From \eqref{longExact}, it follows that $\LC^1_{\n}(\psi_\ell):\LC^1_{\n}(N_{\ell})\to \LC^1_{\n}(S^{b_\ell})$ is an isomorphism. We have shown that $\LC^j_{\n}(\psi_\ell)$ and $\LC^j_{\n}(\phi_{\ell})$ are isomorphisms for all $j\geq 1$. Using the commutativity of \eqref{commdiag} it follows that 
\[
 \LC^j_{\n}(\alpha_\ell)=\LC^j_{\n}(\psi_\ell)\circ \LC^j_{\n}(\phi_{\ell}):\LC^j_{\n}(S^{b_\ell})\to \LC^j_{\n}(S^{b_\ell})
\]
is an isomorphism for all $j\geq 1$. Since $\LC^j_{\n}(-)$ is an $S$-linear functor the map $\LC^j_{\n}(\alpha_\ell)$ is defined by matrix multiplication from the matrix $A_\ell$ applied to the components of $\LC^j_{\n}(S^{b_\ell})$. Since $A_\ell$ has entries in $\n$  it must kill socle elements of $\LC^j_{\n}(S^{b_\ell})$. Therefore  $\LC^j_{\n}(S^{b_\ell})$ has no socle elements. Since 
$\LC^j_{\n}(S^{b_\ell})$ is $\n$-torsion it follows that $\LC^j_{\n}(S^{b_\ell})=0$ for all $j\geq 1$. By \cite[Theorem 9.3]{Iyengar} we get the second equality in the next display.
\[
\dim(R_\q)=\dim(S)=\sup\{ j|\ \LC^j_{\n}(S)\neq 0\}=0
\]
 Thus $\q\in\minn(R)$;   
hence $\minn(\Omega_{2i}(M))\subseteq\minn(R)$ for all $i\gg 0$, and \eqref{thm:main1} follows.
\end{proof}

\begin{syzcor} \label{cor:dim}
\index{Betti number!for local rings, $\beta^R_i(-)$}
Let $R$ be a noetherian local ring and $M$ a finitely generated $R$-module with eventually non-decreasing Betti numbers. Then $(\dim(\Omega_{2i}(M)))_{i=0}^{\infty}$ and $(\dim(\Omega_{2i+1}(M)))_{i=0}^{\infty}$
are constant for $i\gg 0$. If $\pd(M)=\infty$ then one sequence stabilizes to $\dim(R)$ and the other sequence stabilizes to $\dim(R/\p)$ for some $\p\in\minn(R)$.
\end{syzcor}

\begin{proof}
By Theorem \ref{thm:main}~\eqref{thm:main2} both sequences are constant for $i\gg 0$. If $\pd(M)=\infty$ then $\minn(\Omega_i(M))\neq 0$ for all $i$. Therefore by Theorem \ref{thm:main}~\eqref{thm:main1} one sequence will stabilize to $\dim(R/\p)$ and the other to $\dim(R/\q)$ for some $\p,\q\in\minn(R)$. Since 
\[
\supp(\Omega_{2i}(M))\cup\supp(\Omega_{2i+1}(M))=\spec(R),
\]
it follows that $\dim(\Omega_{2i}(M))=\dim(R)$ or $\dim(\Omega_{2i+1}(M))=\dim(R)$.
\end{proof}

Corollary \ref{cor:equidim} follows immediately.
Note that if $R$ is a domain or if $\dim(R)\leq 1$, then $R$ is equidimensional; hence, one can apply Corollary \ref{cor:equidim}.

\begin{syzrmk}\label{rmk2}
It should be noted that \cite[Remark 5.6]{Crabbe} claims that using \cite[Proposition 5.5]{Crabbe} one can show that if $R$ is equidimensional and Conjecture \ref{Q2} is true, then $\dim(\Omega_n(M))$ is constant for $n\gg 0$.  However, \cite[Proposition 5.5]{Crabbe} requires the assumption that $\dim(R)\geq 2$.  Therefore although the conclusions of  \cite[Remark 5.6]{Crabbe} are correct, the justification given for these conclusions is invalid. One should note the justification uses a localization argument, so it is invalid in every positive dimension, not just dimension 1. 
\end{syzrmk}

We now turn our attention to determining how quickly $(\supp(\Omega_{2i}(M)))_{i=0}^{\infty}$ stabilizes once the Betti numbers of $M$ become non-decreasing.

\begin{syzlem}\label{lem:shrink}
\index{Betti number!for local rings, $\beta^R_i(-)$}
Let $R$ be a noetherian local ring and $M$ a finitely generated $R$-module. If $\beta_i(M)=\beta_{i+1}(M)$ for some $i>0$, then 
$\supp(\Omega_i(M))=\supp(\Omega_{i+2}(M))$.

Suppose $\beta_0(M)=\beta_1(M)$. Then we have the following:
\begin{enumerate}[(a)]
\item\label{lem:shrink1} If $\supp(M)\smallsetminus\supp(\Omega_2(M))\neq\emptyset$, then $M$ is not a first syzygy.
\item \label{lem:shrink3} If $\p\in\minn(M)\smallsetminus\supp(\Omega_2(M))$, then $\height(\p)=1$.
\end{enumerate}
\end{syzlem}

\begin{proof}
\index{Betti number!for local rings, $\beta^R_i(-)$}
Suppose $\beta_0(M)=\beta_{1}(M)$ and $\supp(M)\smallsetminus\supp(\Omega_{2}(M))\neq\emptyset$.
Consider the exact sequence 
\[
0\to\Omega_2(M)\to R^{\beta_1(M)} \to R^{\beta_{0}(M)}\to M\to 0.
\]
Choose $\p\in\minn(M)\smallsetminus\supp(\Omega_{2}(M))$. 
Since $M_\p$ has  finite length as an $R_\p$-module, the complex 
\[
0\to R^{\beta_1(M)}_\p \to R^{\beta_{0}(M)}_\p \to 0
\]
has non-zero finite length homology.  By the New Intersection Theorem \cite{roberts} it follows that $\dim(R_\p)\leq 1$. 

Fact \ref{fact:linalg} implies that there exists a minimal $R_\p$-free resolution of $M_\p$ of the form 
\[
0\to R^n_\p\to R^n_\p\to M_\p\to 0
\]
for some $n>0$.  Therefore $1\geq\dim(R_\p)\geq\pd_{R_\p}(M_\p)=1$; hence 
$\height(\p)=\dim(R_\p)=1$, $\depth(M_\p)=0$ and $\depth_{R_\p}(R_\p)=\depth(M_\p)+\pd(M_\p)=1$.

Assume that $M=\Omega_1(L)$ for some $R$-module $L$. We will obtain a contradiction. Since $M_\p$ is finite length and $\dim(R_\p)=1$, it follows that $M_\p$ has no $R_\p$-free summands.  Therefore 
$M_\p=\Omega_1^{R_\p}(L_\p)$.  Since $M_\p$ is a first syzygy, 
\[
0=\depth(M_\p)\geq\min\{1,\depth(R_\p)\}=1.
\]
This is a contradiction; hence $M$ is not a first syzygy. 

Now suppose that $\beta_i(M)=\beta_{i+1}(M)$ for some $i>0$. Since $\Omega_i(M)$ is a first syzygy of $\Omega_{i-1}(M)$  it follows that $\supp(\Omega_i(M))\subseteq\supp(\Omega_{i+2}(M))$. By Lemma \ref{lem:supp} we get the opposite inclusion and the result follows.
\end{proof}

The following is an example where we have $\beta_1(M)=\beta_0(M)$ and $\supp(M)\nsubseteq \supp(\Omega_2(M))$.

\begin{syzex}
Let $S=k[x,y,z]$ and $\m=(x,y,z)$.  Let $R=S_\m/yzS_\m$ and let $M=R/xyR$.  The complex 
\[
\xymatrix{\cdots \ar[r]^z & R \ar[r]^y & R\ar[r]^z & R\ar[r]^{xy} & R\ar[r] & 0}
\]
is a minimal free resolution of $M$.  We have $\Omega_2(M)\cong zR\cong R/(y)$.  The prime ideal $\p=(x,z)$ of height 1 is in $\supp_R(M)$ but it is not in $\supp_R(\Omega_2(M))$. 
\end{syzex}

\begin{syzprop}\label{prop:quick}
\index{Betti number!for local rings, $\beta^R_i(-)$}
 Let $R$ be a noetherian local ring and $M$ a finitely generated $R$-module with non-decreasing Betti numbers.  Then either 
$\supp(\Omega_{2i}(M))$ is constant for all $i\geq 1$, or there exists $n\geq 1$ such that $\supp(\Omega_{2i}(M))=\spec(R)$ for all $i> n$ and $\supp(\Omega_{2j}(M))$ is constant for $1\leq j\leq n$.
\end{syzprop}

\begin{proof}
Suppose $\supp(\Omega_{2i}(M))\neq \spec(R)$ for for some $i\geq 2$.  By Lemma \ref{lem:supp} it follows that $\beta_{2j}(M)=\beta_{2j+1}(M)$ for all $j$ with $0\leq j<i$. From Lemma \ref{lem:shrink} we get that $\supp(\Omega_{2j}(M))$ is constant for $1\leq j\leq i$. 

Now suppose there exists $n$ such that $\supp(\Omega_{2n+2}(M))=\spec(R)$.  Assume that 
$\supp(\Omega_{2n+2i}(M))\neq\spec(R)$ for some $i>0$. Then by Lemma \ref{lem:supp}~\eqref{lem:supp3} it follows that 
\[
\supp(\Omega_{2n+2i}(M))\cap\minn(R)=\supp(\Omega_{2n+2}(M))\cap\minn(R)=\minn(R).
\]
Thus $\supp(\Omega_{2n+2i}(M))=\spec(R)$ a contradiction. Therefore $\supp(\Omega_{2n+2i}(M))=\spec(R)$ for all $i>0$ and the result follows.
\end{proof}

We conclude with a few key examples. 
The following example is due to Hamid Rahmati and can be found in \cite{Crabbe}.
\begin{syzex}
Let $R=k[[x,y]]/(x^2,xy)$ and $M=R/(y)$. A minimal free resolution of $M$ has the form
\[
\xymatrix{ \cdots\ar[rr]^{\left[{{x \ y \ 0}\atop{0 \ 0 \ x}}\right]} && R^2\ar[r]^{[x,y]} & R \ar[r]^{x} & R \ar[r]^{y} & R \ar[r] & 0}.
\]
Also $\dim(M)=\dim(\Omega_2(M))=0$ and $\dim(\Omega_i(M))=1=\dim(R)$ for $i\neq 0, 2$.
\end{syzex}

In the following example we construct a module such that the support of each of its odd syzygies is equal to $\spec(R)$ while the support of each of its even syzygies is not.
\begin{syzex}
Let $R=[a,b,c,d,e]/(ade-bce)$.  Let $M$ be the cokernel of the first map in the following matrix factorization:
\[
 \xymatrix{\cdots\ar[r]^{\left[{{a \ b}\atop{c \ d}}\right]} & R^2\ar[rr]^{\left[{{de \ -be}\atop{-ce \ ae}}\right]} && R^2\ar[r]^{\left[{{a \ b}\atop{c \ d}}\right]} & R^2}.
\]
Then $\supp(\Omega_{2i+1}(M))=\spec(R)$ and $\supp(\Omega_{2i}(M))=\supp(R/(ad-bc))\neq\spec(R)$ for all $i\geq 0$. 
\end{syzex}
The following example is due to Craig Huneke and can be found in \cite{Crabbe}.
\begin{syzex}
Let $S=\mathbb Q[x,y,z,u,v]$ and let $I\subseteq S$ be the ideal
\[
I=\langle x^2,xz,z^2,xu,zv,u^2,v^2,zu+xv+uv,yu,yv,yx-zu,yz-xv \rangle.
\]
Let $R=S/I$, which is a $1$-dimensional ring of depth $0$.  A computation using Macaulay2 \cite{M2} yields that $y$ is a parameter, $(0:y)=(u,v,z^2)$ and $(y)=(0:_R(0:_Ry))$. Let $M$ be the cokernel of the rightmost map in the following exact complex.
\[
\xymatrix{
\cdots \ar[r] & R^3\ar[rr]^{[u\ v\ z^2]} && R\ar[r]^{y}& R\ar[r]^{\left[{u\atop v}\atop{z^2}\right]}& R^3
}
\]
Then the first and
third syzygy modules of $M$ are $R/(y)$ and $(0:y)$ respectively. These are both modules of finite
length since $y$ is a parameter, but all other syzygies have dimension 1.
\end{syzex}

\section*{Acknowledgments}
We are grateful to our other MRC Snowbird group members, Christine Berkesch and Daniel Erman for their hard work and ideas during the summer research community.  Craig Huneke guided our MRC Snowbird group and developed many of the key arguments in the proof of our main result. Additionally, Srikanth Iyengar showed us how to simplify our main argument.
We would also like to thank Luchezar Avramov, David Jorgensen, Dan Katz, 
Roger Wiegand and the anonymous referee for giving us invaluable feedback on this research.  Lastly we would like to thank Roger and Sylvia Wiegand for their hospitality during Kristen's visit to Lincoln.

\end{document}